\documentclass[12pt]{article}
\usepackage{amssymb,amsthm,amsmath,amsfonts}
\usepackage[utf8]{inputenc}
\usepackage{enumerate, tikz}

\newtheorem{lemma}{Lemma}
\newtheorem{theorem}[lemma]{Theorem}
\newtheorem{prop}[lemma]{Proposition}
\newtheorem{corollary}[lemma]{Corollary}

\newtheorem{remark}[lemma]{Remark}

\newcommand{\Prob}{\mathbb{P}}

\newcommand{\E}{\mathbb{E}}
\newcommand{\Var}{\mathrm{Var}}
\newcommand{\Cov}{\mathrm{Cov}}
\newcommand{\R}{\mathbb{R}}
\newcommand{\vol}{\textrm{Vol}}

\newcommand{\Ent}{\mathrm{Ent}}

\newcommand{\al}{\alpha}
\newcommand{\la}{\lambda}

\newcommand{\sk}[1]{\left\langle #1 \right\rangle}

\begin{document}
\title{Thin-shell concentration for convex measures}
\author{Matthieu Fradelizi,  Olivier Gu\'edon and Alain Pajor}

\newcommand\address{\noindent\leavevmode\noindent
{\small
Matthieu Fradelizi, Olivier Gu\'{e}don, Alain Pajor\\
Universit\'{e} Paris-Est \\
Laboratoire d'Analyse et Math\'{e}matiques Appliqu\'ees (UMR 8050) \\
UPEMLV, F-77454, Marne-la-Vall\'ee Cedex 2, France  \\
\texttt{\small olivier.guedon@u-pem.fr, matthieu.fradelizi@u-pem.fr,
alain.pajor@u-pem.fr}}}

\date{}

\maketitle

%% Classification and key words; note that the 2010 classification is used:

\renewcommand{\thefootnote}{}



\renewcommand{\thefootnote}{\arabic{footnote}}
\setcounter{footnote}{0}

%%%%%%%%

\begin{abstract}
We prove that for $s<0$, $s$-concave measures on $\R^n$ satisfy a
thin-shell concentration  similar to the log-concave case. It leads to a Berry-Esseen type estimate for most of their one dimensional marginal distributions. We also establish sharp reverse H\"older inequalities for $s$-concave measures.
\end{abstract}

\section{Introduction}

For any subsets $A,B\subset \R^n$, the Minkowski sum is defined by
\[ A + B=\{ a + b\,:\, a\in A, b\in B\}.
\]
Let $s \in[-\infty, 1]$. A measure $\mu$ on $\R^n$  is called $s$-concave
whenever
\[
\mu \left( (1-\la) A + \la B \right)
\ge
\left( (1-\la) \mu(A)^s + \la \mu(B)^s
\right)^{1/s},
\]
for every $\lambda\in [0,1]$ and  every compact subsets $A, B \subset \R^n$ such that 
$\mu(A) \mu(B)>0$.
When $s=0$, this inequality
should be read as
\[
\mu \left( (1-\la) A + \la B \right)
\ge
 \mu(A)^{1-\la} \mu(B)^{\la}
\]
 and it defines $\mu$ as a log-concave measure. When $s=-\infty$,  the measure is said to be convex and the inequality is replaced by
\[
\mu \left( (1-\la) A + \la B \right)
\ge
\min\left(\mu(A),\mu(B)
\right).
\]

Notice that the class of  $s$-concave measures on $\R^n$ is decreasing in $s$ so that any
$s$-concave measure is a convex measure.
Any $s$-concave measure
with $s\geq 0$ is log-concave and the thin-shell concentration for log-concave measures has been studied in \cite{Fleury, FGP, GM, Klartag1, Klartag2}.
 The purpose of this paper is to prove a thin-shell concentration for  $s$-concave measures in the case $s<0$, which we consider
from now on. By measure, we always mean probability measure.

The class of $s$-concave measures was introduced and studied in
\cite{Bo1,Bo2}, where a complete characterization was established. An $s$-concave
measure is supported on some convex subset of an affine subspace
where it has a density (see Section \ref{sec:tools} for more details).  When the support of an $s$-concave
measure $\mu$ generates the
whole space, we say that $\mu$ is full-dimensional.

A random vector with an $s$-concave
distribution is called $s$-concave.  The linear image of  an $s$-concave
random vector is also $s$-concave.  We say that a random vector is full-dimensional if its distribution is full-dimensional. It is known that any semi-norm
of an $s$-concave random vector with $s< 0$ has moments of all order $p\in (0,-1/s)$ (see \cite{Bo1} and \cite{AGLLOPT}). The Euclidean norm of an $s$-concave random vector $X$  has a finite moment of order 2 if and only if $s>-1/2$. Since we are interested in comparison of moments of the Euclidean norm
with the moment of order 2, we will always assume that $-1/2<s< 0$.

Let $n\geq 1$ be an integer. The Euclidean space $\R^n$ is equipped with its Euclidean norm
$|\,.\,|_2$  and its scalar product $\langle \, .\,,\,.\,\rangle$. Its unit sphere is denoted by $S^{n-1}$ and its unit ball by $B_2^n$.
We say that a random vector $X$ is isotropic if $\E X = 0$ and for every $\theta \in S^{n-1}$,
$\E \langle X,\theta \rangle^2 = 1$.  Observe that if $X$ is an $s$-concave full-dimensional random vector and   $-1/2<s$, we can always find an affine transformation $A$ such that $AX$ is isotropic.

 Let $p\in \R$ and  $X\in\R^n$ be a random vector. Assume that $|X|_2$ has finite moments of order $2$ and $p$ with the convention that $(\E |X|_2^p)^{1/p}= \exp(\E \ln |X|_2)$ for $p=0$. We define
 \[
\alpha_p(X):= \left | \frac{(\E |X|_2^p)^{1/p}}{(\E |X|_2^{2})^{1/2}}  - 1 \right| .
\]

Our main result is the following
\begin{theorem}
\label{th:moments}
Let $r > 2$.
Let $X\in\R^n$ be a full-dimensional $(-1/r)$-concave random vector.

If $X$  is isotropic, then  for any $p$ such that
$|p|\le c \min(r, n^{1/3})$, we have
\[
\alpha_p(X) \le   \frac{C\, |p-2|}{r} + \left(\frac{C \, |p-2|}{n^{1/3}}\right)^{3/5},
\]
where $C$ and $c$ are universal constants.

In the general case (when $X$ is not isotropic), let $A$ be an affine transformation such that $AX$ is full-dimensional and isotropic.  Then  for any $p\in\R$ such that
$|p|\le c \min\left(r, \frac{n^{1/3}}{\|A\|^{2/3} \|A^{-1}\|^{2/3}}\right)$, we have
\[
\alpha_p(X) 
 \le   \frac{C\, |p-2|}{r} + \left(\frac{C \, |p-2| (\|A\| \|A^{-1}\|)^{2/3}}{n^{1/3}}\right)^{3/5},
\]
where $C$ and $c$ are universal constants.
\end{theorem}

We also show (see Remark \ref{final remark}) that for $r > n + \sqrt n$, the estimate of $\alpha_p(X)$ in Theorem \ref{th:moments} can be improved  and recovers the estimate of the log-concave case from \cite{GM}.
\smallskip

To present the connections between moment inequalities, concentration in a thin-shell and the Berry-Esseen theorem for one  dimensional marginals, let us introduce some notations.

Let $X\in \R^n$ be an isotropic random vector. Thus $\E |X|_2^2=n$. Define
$\varepsilon(X)$ to be the smallest number $\varepsilon>0$ such that
\begin{equation}
\label{equ:concentration}
\Prob\left(
\left|  \frac{ |X|_2}{\sqrt n}-1\right |\geq \varepsilon\right)
\leq \varepsilon.
\end{equation}
If $\varepsilon(X)=o(1)$ with respect to the dimension $n$, we say that $X$ is concentrated in a thin-shell.
This is the usual jargon of the subject. More rigorously, it suggests that we are considering a sequence of random vectors $(X_n)$ with $X_n\in\R^n$ and that  $\varepsilon(X_n)=o(1)$ as $n$ goes to $\infty$.
It was shown in \cite{ABP} (see also \cite{BV1, BHVV}) that if
an isotropic random vector $X$ uniformly distributed on a convex body in $\R^n$ is such that $\varepsilon(X)=o(1)$, then almost all  one dimensional marginal distributions of $X$ satisfy a Berry-Esseen theorem.
More generally, let $X\in\R^n$ be an isotropic random vector, it was proved in \cite{Bo2003} that
$$\sigma_{n-1}\left( \theta\in S^{n-1}\,:\,
\sup_{t\in \R}\left|
\Prob(\langle X,\theta\rangle\leq t)-\Phi(t)
\right|\geq 4\varepsilon(X)+\delta
\right)\leq 4n^{3/8}e^{-cn\delta^ 4},$$
where $\sigma_{n-1}$ denotes the rotation
invariant probability measure on the unit sphere $S^{n-1}$, $\Phi$  is the standard normal distribution function and
$c>0$ is a universal constant.
It is worth noticing that the result from \cite{Bo2003} does not assume log-concavity.
Assuming only that $X$ is isotropic, we get that if $\varepsilon(X)$ is $o(1)$ then almost all the one dimension marginal distributions of $X$ are approximately Gaussian.
The fact that indeed for all log-concave random vector $\varepsilon(X)=o(1)$ was proved later in \cite{Klartag1, FGP}  and the best
estimate at this date \cite{GM}  is that
$$\varepsilon(X) = O(n^{-1/6} \log n).$$

Now  let $p>2$ and assume that $X$ is isotropic and that $|X|_2$ has a finite  moment of order $p$.
Then  $\varepsilon(X)$  is $o(1)$ if and only if
 $\alpha_p(X)$  is $o(1)$, see Remark~\ref{remark: epsilon alpha} below.
 Hence Theorem \ref{th:moments} ensures that if $r\to+\infty$ with the dimension $n$ then any isotropic $(-1/r)$-concave random vector satisfies a  thin-shell concentration and therefore almost all its one dimensional marginals verify a Berry-Esseen theorem. As a matter of fact, this condition on $r$ is necessary. If $r$ is fixed and does not depend
on the dimension $n$, Proposition~\ref{prop:example}  gives an example of  an isotropic ($-1/r$)-concave random vector $X\in\R^n$
which does not satisfy a thin-shell concentration. Remark \ref{rem:asymptotic} also shows the asymptotic sharpness of Theorem \ref{th:moments}, since for this example,  for a fixed $p>2$, $\alpha_p(X)\geq C(p-2)/r$ for $r$ and $n$ large enough, where $C>0$ is a universal constant.

\smallskip

To build the proof of Theorem \ref{th:moments}, we need
to extend to the case of $s$-concave measures several tools coming from the study of log-concave measures. This is the purpose of Section \ref{sec:tools}. Some of them were already achieved by Bobkov \cite{Bobkov}, like analog of the Ball's bodies \cite{ball} in the $s$-concave setting. Some others were also noticed previously (see e.g. \cite{Bobkov}, \cite{AGLLOPT}) but not with the most accurate point of view. These new ingredients are analog to the results of  \cite{Bo3} in the log-concave setting and are at the heart of our proof.
As in the approach of  \cite{Fleury} or  \cite{GM}, an important ingredient  is  the log-Sobolev inequality on $SO(n)$.  
It follows e.g. from the work of Bakry and \'Emery \cite{BE} and the calculus of the Ricci curvature of $SO(n)$ (see \cite[Formula (F6)]{GZ} for example) that for any Lipschitz function $f : SO(n) \to \R^+$ (see sections \ref{sec:thinshell} and \ref{sec:appendix} for definitions)
\begin{equation}
\label{eq:LogSob-vrai}
\E(f(U)\log f(U)) - \E f(U) \log (\E f(U)) \le \frac{c}{n} \E \left( |\nabla \log f(U)|^2 f(U) \right),
\end{equation}
where $U$ is uniformly distributed on $SO(n)$.
It allows to get reverse H\"older inequalities (see inequality $(15)$ in \cite{Fleury}): for every $f : SO(n) \to \R$, let  $L$ be the log-Lipschitz constant of $f$ (that is the Lipschitz constant of $\log f$), then for every $q>r>0$,
\begin{equation}
\label{eq:reverseLogSob}
(\E |f(U)|^q)^{1/q} \le \exp\left( \frac{c \, L^2}{n} \,(q-r) \right) \ (\E |f(U)|^r)^{1/r},
\end{equation}
where $U$ is uniformly distributed on $SO(n)$.

Let $X$ be a $(-1/r)$-concave random vector in $\R^n$ with full-dimensional support 
and distributed according to a measure with a density function $w : \R^n \to \R_+$. 
For any linear subspace $E$, denote by $P_E$ the orthogonal projection onto $E$ and for any $x\in E$ denote by
\[
\pi_Ew(x)=\int_{x+E^\bot}w(y)dy
\]
the marginal of $w$ on $E$.
Given an integer $k$ between $1$ and $n$, a real number $p \in (-k, r)$, a linear subspace $E_0$ of $\R^n$ of dimension $k$ and $\theta_0 \in S(E_0)$, where
$S(E_0)$ denotes the unit sphere of $E_0$,
we define the function $h_{k,p} : SO(n) \rightarrow \R_+$ by
\begin{equation} \label{eq:h_p}
 h_{k,p}(u) := |S^{k-1}| \int_{0}^\infty t^{p+k-1} \pi_{u(E_0)} w(t u(\theta_0)) dt,
\end{equation}
for every $u\in SO(n)$, where $|S^{k-1}|$ denotes the area of the sphere.

Following the approach of \cite{Klartag2, Fleury}, we observe that
for any $p \in (-k,r)$
\begin{equation} \label{eq:so1}
\E |X|_2^p = \frac{\Gamma((p+n)/2)\Gamma(k/2)}{\Gamma(n/2)\Gamma((p+k)/2)} \E h_{k,p}(U),
\end{equation}
where $U$ is uniformly distributed on $SO(n)$.
In view of $(\ref{eq:so1})$ and the definition of $h_{k,p}$, we notice that it is  of importance to work with family of measures which are stable after taking the marginals and it is clear from the definition that 
for any subspace $E$,  if $X$ is $(-1/r)$-concave, then $P_E X$ is also $(-1/r)$-concave.

In the next section \ref{sec:tools}, we first introduce more notation and recall important facts concerning convex measures.  Then we give an example of an isotropic ($-1/r$)-concave random vector $X\in\R^n$ that does not satisfy a thin-shell concentration, when $r$ is fixed with respect to the dimension. Finally, we extend  to the case of $s$-concave measures several tools coming from the study of log-concave measures that will be essential in the proof of Theorem \ref{th:moments}.
Section \ref{sec:thinshell} is devoted to the proof of Theorem \ref{th:moments}. Some of the results of these two sections are  either classical or  variation of known results; their proofs are shifted to the appendix.

{\bf Acknowledgement}. We deeply and warmly thank the referee for his constructive comments on the first version of the paper. It  forced us to clarify several main points
and we hope that it  improved the presentation of the paper.

\section{Preliminary results for $s$-concave measures}
\label{sec:tools}

We first recall some properties of $s$-concave measures and their relation to $\beta$-concave functions.

The class of $s$-concave measures was introduced and studied in
\cite{Bo1,Bo2}, where the following complete characterization was established. An $s$-concave
measure $\mu$ on $\R^n$ is supported on some convex subset of an affine subspace
where it has a density.  When this subspace is the whole space, we say that $\mu$ is full-dimensional. In this case,
its density $w$ is $\beta$-concave with $\beta=s/(1-ns)$. Recall that 
a function $f:\R^n\to \R_+$  is called $\beta$-concave 
whenever
\[
f\left( (1-\la) x + \la y \right)
\ge
\left( (1-\la) f(x)^\beta + \la f(y)^\beta
\right)^{1/\beta}
\]
for every $\lambda\in [0,1]$ and  every $x,y \in \R^n$ such that $f(x)f(y)>0$, where the right hand side is replaced by $f(x)^{1-\la}f(y)^\la$ for $\beta =0$. Note that when $\beta<0$ which will be the case below, $\beta$-concavity means that  $f^{\beta}$ is convex on its convex support $\{f>0\}$. 

%A random vector with an $s$-concave
%distribution is called $s$-concave.  The linear image of  an $s$-concave
%random vector is also $s$-concave.  We say that a random vector is full-dimensional if its distribution is full-dimensional. 
%It is known that any semi-norm
%of an $s$-concave random vector with $s< 0$ has moments of all order $p\in (0,-1/s)$ (see \cite{Bo1} and \cite{AGLLOPT}). The Euclidean norm of an $s$-concave random vector $X$  has a finite moment of order 2 if and only if $s>-1/2$. Since we are interested in comparison of moments of the Euclidean norm
%with the moment of order 2, we will always assume that $-1/2<s< 0$.

We will use a similar language for probability measure, random vector and function which are related here as distribution, law of a random vector and density of probability. It is important to remember that when $X\in\R^n$ is $(-1/r)$-concave full-dimensional, then the result recalled above states that its distribution has a support that generates $\R^n$ and has a density which is $(-1/(n+r))$-concave.

Recall that for every $x>0$, $\Gamma(x)=\int_0^\infty u^{x-1}e^{-u}\,du$ and
for every $x,y>0$, $
B(x ,y) = \int_{0}^{1} u^{x-1} (1- u)^{y-1} du = \int_0^{+\infty} u^{x - 1} (u+1)^{-(x+y)} \,du.
$

The following inequality  of  Paley-Zygmund  type is well known.
%\medskip\hrule\medskip

 \begin{lemma}
 \label{lemma:Paley-Zygmund}
Let $2<p<s$. Let $Y$ be a non-negative random variable with finite $s$-moment. Then for every  $0\leq t\leq (\E Y^p)^{1/p}$ we have
\begin{equation*}\label{Paley-Zygmund}
 \Prob(Y\geq t)
 \geq \left(\frac{\E Y^p-t^p}{ (\E Y^s)^{p/s}}\right)^{s/(s-p)}.
 \end{equation*}
\end{lemma}

\begin{proof}
Using H\"older inequality, we have
$$\E Y^p= \E Y^p1_{Y<t}+ \E Y^p1_{Y\geq t}\leq t^p+(\E Y^s)^{p/s}\Prob(Y\geq t)^{1-p/s}.$$
Thus
$$
 \Prob(Y\geq t)\geq \left(\frac{\E Y^p-t^p}{ (\E Y^s)^{p/s}}\right)^{s/(s-p)}.$$
\end{proof}

\begin{prop}\label{prop:equiv epsilon}
Let $2<p<s$.  Let $X\in\R^n$ be an isotropic  random vector such that $|X|_2$ has a finite $s$-moment. Then
$$\min\left(\frac{\alpha_p(X)}{2}, \left(
\frac{p\alpha_p(X)/2}{(\alpha_s(X)+1)^{p}}\right)^{s/(s-p)}\right)\leq \varepsilon(X)\leq
\left((\alpha_p(X)+1)^p-1\right)^{1/3}  .$$
\end{prop}

\begin{proof}
Let $\varepsilon>0$.  Applying Lemma \ref{lemma:Paley-Zygmund} to $Y=|X|_2/(\E |X|_2^2)^{1/2}$,
$t=\varepsilon+1$ and noticing that
$\E Y^p= (\alpha_p(X)+1)^p$, $\E Y^s= (\alpha_s(X)+1)^s$,  we get that
$$\Prob\left( \frac{|X|_2}{(\E |X|_2^2)^{1/2}}\geq 1+\varepsilon\right)
\geq
  \left(\frac{(\alpha_p(X)+1)^p -(\varepsilon+1)^p}{(\alpha_s(X)+1)^{p}}\right)^{s/(s-p)}$$
whenever $0<\varepsilon\leq  \alpha_p(X)$.
Since  for $p\geq 1$ and $x\geq y\geq 1$, $x^p-y^p\geq p(x-y)$, we have
$$\Prob\left( \frac{|X|_2}{(\E |X|_2^2)^{1/2}}\geq 1+\varepsilon\right)\geq
  \left(
\frac{p(\alpha_p(X)-\varepsilon)}{(\alpha_s(X)+1)^{p}}\right)^{s/(s-p)}.$$
Therefore
$$\Prob\left( \frac{|X|_2}{(\E |X|_2^2)^{1/2}}\geq 1+\varepsilon\right)\geq
  \left(
\frac{p\alpha_p(X)/2}{(\alpha_s(X)+1)^{p}}\right)^{s/(s-p)}$$
whenever $0<\varepsilon\leq  \alpha_p(X)/2$.
The left-hand side inequality follows.

Since for $q\geq 1$, $|x-1|\leq |x^q-1|$ for every $x\geq 0$, Markov inequality gives
\[\Prob\left(
\left| \frac{|X|_2}{(\E |X|_2^2)^{1/2}}-1\right |\geq \varepsilon\right)
\leq \Prob\left(
\left| \frac{|X|_2^q}{(\E |X|_2^2)^{q/2}}-1\right|\geq \varepsilon\right)
\leq \frac{\E \left| \frac{|X|_2^q}{(\E |X|_2^2)^{q/2}}-1\right|^2}{\varepsilon^2}.
\]
To conclude the right-hand side inequality, take $q=p/2$ and observe that
$$\E \left| \frac{|X|_2^q}{(\E |X|_2^2)^{q/2}}-1\right|^2=(\alpha_{2q}(X)+1)^{2q}+1-2(\alpha_q(X)+1)^q\leq (\alpha_{2q}(X)+1)^{2q}-1.$$
\end{proof}

\begin{remark}\label{remark: epsilon alpha} Let $2<p<s$.  Let $X\in\R^n$ be an isotropic  random vector such that $|X|_2$ has a finite $s$-moment. Proposition \ref{prop:equiv epsilon} shows that
$\varepsilon(X)$  is $o(1)$ if and only if
$\alpha_p(X)$  is $o(1)$ when $n\to\infty$.
\end{remark}

Now we estimate $\varepsilon(X)$ for an example which shows that
an isotropic ($-1/r$)-concave random vector $X\in\R^n$
may not satisfy a thin-shell concentration. In the proposition below, the notation $\varliminf$ refers to the limit inferior.

\begin{prop}\label{prop:example}
Let $r>2$. There exists a sequence $(X_n)_n$ of isotropic ($-1/r$)-concave random vectors $X_n\in\R^n$ such that
$$\varliminf_{n\to\infty}\varepsilon(X_n)\geq c(r)>0,$$
where $c(r)>0$ depends only on $r$. \end{prop}
\begin{proof}
Let $r>2$ and $2<p<r$ and
let $X_n\in\R^n$ be an isotropic random vector with density
$$f_{n,r}(x)=\frac{c_1}{(1+c_2 |x|_2)^{r+n}}, $$
where $c_1$ and $c_2$ are normalization factors.
From \cite{Bo1,Bo2}, such a random vector is $(-1/r)$-concave.
An immediate computation gives that
$$
\frac{(\E|X_n|_2^p)^{1/p}}{(\E|X_n|_2^2)^{1/2}}=
\left(\frac{B(n+p,r-p)}{B(n,r)}\right)^{1/p}
\left(\frac{B(n+2,r-2)}{B(n,r)}\right)^{-1/2}.$$
For fixed $r$ and $2<p<r$, we have
\begin{equation}
\label{equ:moment example}
\lim_{n \to + \infty}
\frac{(\E|X_n|_2^p)^{1/p}}{(\E|X_n|_2^2)^{1/2}} =
\left(\frac{\Gamma(r-p)}{\Gamma(r)}\right)^{1/p}
\left(\frac{\Gamma(r-2)}{\Gamma(r)}\right)^{-1/2}
\end{equation}
and by the strict log-convexity of the Gamma function, we have
\[
\lim_{n \to + \infty}\left(\alpha_p(X_n) +1\right)=\lim_{n \to + \infty}\frac{(\E|X_n|_2^p)^{1/p}}{(\E|X_n|_2^2)^{1/2}}
>1.
\]
As a consequence for any $2<p<r$, $\lim_{n \to + \infty}\alpha_p(X_n)>0.$

Now let $2<p<s<r$. From  Proposition \ref{prop:equiv epsilon}, we get
\begin{equation}
\label{lowerboundepsilon}
\varliminf_{n \to + \infty} \varepsilon(X_n)\geq
\lim_{n \to\infty}
\min\left(\frac{\alpha_p(X_n)}{2},\left(\frac{p\alpha_p(X_n)/2}{(\alpha_s(X_n)+1)^{p}}\right)^{s/(s-p)}\right)>0.
\end{equation}
Choose  $p=(2+r)/2$ and $s=(p+r)/2$ for which  $2 < p < s<r$
and note that
the right hand side term in equation (\ref{lowerboundepsilon}) depends only on $r$.
This concludes the proof.

\end{proof}

\begin{remark}\label{rem:asymptotic}
Let $2<p<r$ and let $r\to\infty$. Applying  Stirling formula in  (\ref{equ:moment example}) when $r\to\infty$, a calculation gives that
\[
\lim_{r \to \infty} r\lim_{n \to \infty} \alpha_p(X_n)=(p-2)/2.
\]
This asymptotic estimate  shows that for a fixed $p>2$ and $r$ and $n$ large enough, then
$\alpha_p(X_n)\geq C(p-2)/r$ where $C>0$ is a universal constant. This proves  the sharpness of Theorem \ref{th:moments} under these conditions.
\end{remark}

We now prove some inequalities for $s$-concave measures that will be useful tools in the next section.

\begin{theorem}
\label{th:logconcave}
(1) Let $f:[0,\infty) \to [0,\infty)$ be a measurable function such that $\|f\|_\infty>0$. Then
\[ p\mapsto \left(\int_0^\infty pt^{p-1} f(t)\,dt   /  \|f\|_\infty\right)^{1/p}\]
is non-decreasing on its domain of definition.

(2) Let $\al >0$ and $f:[0,\infty) \to [0,\infty)$ be $(-1/\al)$-concave, continuous and integrable. Define $H_f:[0,\al) \to \R_+$ by
\[
    H_f(p) =  \left\{  \begin{array}{ll}
    \displaystyle \frac{1}{B(p, \al -p)} \,  \int_{0}^{+\infty} t^{p-1} f(t) dt    &\text{for}\ 0<p<\alpha
    \\
    \\
    f(0) & \text{for}\  p=0.  \end{array} \right.
\]
Then $H_f$ is log-concave on $[0, \al)$.
\end{theorem}
The proof of the first part may be treated as in Lemma 2.1 in \cite{MP} and the proof  of the second part is identical to the well known $(1/n)$-concave case \cite{Bo3}. We postpone the proof of Theorem \ref{th:logconcave} to the appendix.

We present several consequences of this result such as  some reverse H\"older inequalities with sharp constants in the spirit of Borell's \cite{Bo3} and
 Berwald's  \cite{Berwald} inequalities.

\begin{corollary}
\label{cor:Khinchine}
Let $r>0$ and $\mu$ be a $(-1/r)$-concave measure on $\R^n$. Let  $\phi : \R^n \to \R_+$ such that $\{\phi>0\}$ is convex and $\phi$ is concave on $\{\phi>0\}$. Then the function
\[
    p\mapsto \left\{  \begin{array}{ll}
    \displaystyle\frac{1}{pB(p, r -p)} \, \int \phi(x)^p d\mu(x)   &\text{for}\ 0<p<r
    \\
    \\
    \mu(\{\phi>0\}) & \text{for}\  p=0  \end{array} \right.
\]
is log-concave on $[0, r)$.

Moreover, if  $\mu(\{\phi>0\})>0$ then
for any $0 < p \leq q < r$,
\[
    \left(  \int_{\R^n} \phi(x)^q \frac{ d\mu(x)}{\mu(\{\phi>0\})}    \right)^{1/q}
    \leq
    \frac{\left(q B(q, r-q)\right)^{1/q}}{\left(p B(p, r-p)\right)^{1/p}} \left( \int_{\R^n} \phi(x)^p  \frac{ d\mu(x)}{\mu(\{\phi>0\})}    \right)^{1/p} .
\]
\end{corollary}

\begin{proof}
By the concavity of $\phi$,
for every $u,v \ge 0$ and every $\la \in [0,1]$
 \[
(1-\la) \{   \phi > u \} + \la \{  \phi > v \}
\subset
\{   \phi > (1-\la)u + \la v \}.
 \]
By the $(-1/r)$-concavity of $\mu$, the function $f(t) = \mu(\{ \phi > t\})$ is $(-1/r)$-concave and it is clearly continuous on $\R_+$.  Observe by Fubini that for any $p > 0$,
\[
\int_{\R^n} \phi(x)^p d \mu(x) = \int_0^{+\infty} p t^{p-1} f(t) dt.
\]
The result follows from the part (2) of Theorem \ref{th:logconcave}. The moreover part follows from the log-concavity since then  $p \mapsto (H_f(p)/f(0))^{1/p}$ is a non-increasing function. \end{proof}
The second corollary concerns the function $h_{k,p}$ defined in (\ref{eq:h_p}).
\begin{corollary}
\label{cor:logconcavity}
Let $r>0$ and $u\in SO(n)$. For any $(-1/(r+n))$-concave function $w : \R^n \to \R_+$  and any subspace $E_0$ of dimension $k\le n$, the function
\[
p \mapsto
\left\{
\begin{array}{ll}
\displaystyle
\frac{ h_{k,p}(u)  }{B(p+k, r -p )  } &\ \text{for}\ p  > -k +1
\\
\\
|S^{k-1}| \pi_{u(E_0)} w(0)&\  \text{for}\  p = - k + 1
\end{array}
\right.
\]
is log-concave on $[-k + 1, r)$.
\end{corollary}
\begin{proof}
Since $w$ is $(-1/(r+n))$-concave, we note that $t \mapsto \pi_{U(E_0)} w(t u(\theta_0))$ is $(-1/(r+k))$-concave and it is clearly continuous on $\R_+$.  Theorem \ref{th:logconcave} proves the result.
\end{proof}

We finish with some geometric properties of a family of  bodies introduced by K. Ball in \cite{ball} in the log-concave case.
\begin{corollary}
\label{cor:geoKr}
Let $\alpha>0$. Let $w: \R^n \to \R_+$ be a $(-1/\al)$-concave function such that $w(0) > 0$. For  $0<a <\alpha$ let
\[
K_a(w) = \left\{ x\in\R^n ; \  a \int_0^{+\infty} t^{a-1} w(tx) dt   \ge w(0)
\right\}.
\]
Then for any $0 < a \le b < \al$
\[
\left(\frac{w(0)}{\|w\|_{\infty}}\right)^{\frac{1}{a} - \frac{1}{b}} K_a(w) \subset K_b(w) \subset \frac{(bB(b, \al -b))^{1/b}}{(aB(a, \al - a))^{1/a}} \ K_a(w).
\]
\end{corollary}
\begin{proof} Notice that the sets $K_a$ are star-shaped with respect to the origin,
that is  for every $x\in K_a$  and every $\lambda \in [0,1]$, $\lambda x\in K_a$. The radial function of $K_a$ is
\[
\rho_{K_a}(x):=\sup\{r \,:\, rx\in K_a\}=\left(a\int_0^{+\infty}t^{a-1}\frac{w(tx)}{w(0)}dt\right)^\frac{1}{a}.
\]
For any $x \in \R^n$,
let $f$ be the continuous $(-1/\al)$-concave function defined on $\R^+$ by $f(t) = w(tx)/w(0)$. By (1) of Theorem \ref{th:logconcave}, the function $a\mapsto \left(\int_0^{+\infty}t^{a-1}\frac{f(t)}{\|f\|_\infty}dt\right)^\frac{1}{a}$ is non-decreasing. The left hand side inclusion follows. Moreover, from (2) of Theorem  \ref{th:logconcave}, the function
$H_f: [0,\al) \to \R_+$ is log-concave on $[0, \al)$ with $H_f(0)= 1$. For $0 < a \le b < \al$, we have thus
$H_f(b)^{1/b} \le H_f(a)^{1/a}$. The right hand side inclusion follows.
\end{proof}

\section{Thin shell for convex measures}
\label{sec:thinshell}
The purpose of this section is to prove Theorem \ref{th:moments}. We follow the strategy of the log-concave case initiated in \cite{Klartag1, FGP, Klartag2} and further developed in \cite{Fleury, GM}.
\\
The support function $h_K$ of a non-empty compact set $K\subset\R^n$ is defined by
$$\forall\theta\in\R^n,\quad h_K(\theta)=\sup_{x\in K} \sk{x, \theta}.$$
To any random vector $X$ in $\R^n$ and any $p\ge 1$, we associate its $Z_p^+$-body defined by its support function
\[
\forall\theta\in\R^n,\quad h_{Z_p^+(X)}(\theta) = \left( \E \sk{X, \theta}_+^p \right)^{1/p}.
\]
When the distribution of $X$ has a density $g$, we write
$Z_p^+(g)=Z_p^+(X)$.
Extending a theorem of Ball  \cite{ball} for log-concave functions, Bobkov proved in \cite[Remark 2.6]{Bobkov} (see also \cite[Theorem\ 3.1]{Frad-Cor-Piv}) that if $w$ is  $(-1/(r+n))$-concave on $\R^n$ such that $w(0) > 0$, then
\begin{equation}
\label{def:K_a}
K_a(w)\ \text{is convex and compact for any}\ 0<a \le r+n-1.
\end{equation}
In the case of log-concave measures \cite{Paouris1, Paouris2, GM, GuedonPologne},
several relations between the $Z_p^+$ bodies and the convex sets $K_a$  are known. We need their analogue in the setting of $s$-concave measures for negative $s$.
We start with two technical lemmas. We postpone their proofs to the appendix.

\begin{lemma}
\label{lem:fonctionBeta}
Let  $x, y \ge 1$, then
\begin{equation}
\label{eq:Beta1}
c \, \frac{x}{x+y} \le \left( x B(x, y)\right)^{1/x}  \le C\, \frac{x}{x+y},
\end{equation}
where $c,C$ are positive universal constants.
%Let  $m + p \ge 1$ and $r-p \ge 1$, then
%\begin{equation}
%\label{eq:Beta1}
%c \, \frac{m+p}{m+r} \le \left( (m+p) B(m+p, r-p)\right)^{1/m+p}  \le C\, \frac{m+p}{m+r}
%\end{equation}
Moreover, for $k,r>1$, the extension by continuity at $0$ of the function $p\mapsto\frac{1}{p} \log \frac{B(k+p, r-p)}{B(k,r)}$ is differentiable on $[-(\frac{k-1}{2}),\frac{r-1}{2}]$ and satisfies
\begin{equation}
\label{eq:Beta2}
0 \le \frac{d}{dp} \left( \frac{1}{p} \log \frac{B(k+p, r-p)}{B(k,r)} \right) \le  \frac{1}{r-1} + \frac{1}{k-1} 
\end{equation}
for $p\in[-(\frac{k-1}{2}),\frac{r-1}{2}]$.\end{lemma}
In this paper, we use the notion of geometric distance between sets, defined for every compact subsets $K$, $L \subset \R^n$ containing $0$ in their interior by
\[
d(K, L) = \inf \{ t_2 / t_1 : t_1 L \subset K \subset t_2 L, t_1, t_2 > 0 \}.
\]
Let $n\ge1$, $r\ge2$ and $w$ be the $(-1/(r+n))$-concave density of a probability measure $\mu$ on $\R^n$. Then by Corollary \ref{cor:Khinchine} and Lemma \ref{lem:fonctionBeta}, for $1\le p\le q\le r-1$, one has
\[
Z_p^+(w)\subset Z_q^+(w)\subset c\frac{q}{p}\left(\inf_{\theta\in S^{n-1}} \mu\left(\{x: \langle x, \theta\rangle>0\}\right)\right)^{\frac{1}{q}-\frac{1}{p}}Z_p^+(w).
\]
Fix $\theta\in S^{n-1}$ and define $F(t)=\mu(\{x: \langle x, \theta\rangle\le t\})$, for $t\in\R$. One has $\int_\R tF'(t)dt=\int_{\R^n}\langle x,\theta \rangle w(x)dx=0$ and $F$ is $(-1/r)$-concave. Using Jensen's inequality, we get
\[
F(0)^{-\frac{1}{r}}=F\left(\int_\R tF'(t)dt\right)^{-\frac{1}{r}}\le\int_\R F(t)^{-\frac{1}{r}}F'(t)dt=\left[\frac{F(t)^{1-\frac{1}{r}}}{1-\frac{1}{r}}\right]_{-\infty}^{+\infty}=\frac{1}{1-\frac{1}{r}}.
\]
Hence $\mu(\{x: \langle x, \theta\rangle>0\})\ge\left(1-\frac{1}{r}\right)^r\ge 1/4$ for $r\ge 2$. We have recovered here in a simple way a Gr\"unbaum's type inequality for convex measures due to Bobkov \cite[Theorem~5.2]{Bobkov}. We deduce that, for $1\le p\le q\le r-1$,
\begin{equation}\label{eq:Zp-Zq}
Z_p^+(w)\subset Z_q^+(w)\subset C\frac{q}{p}Z_p^+(w) \quad{\rm and}\quad d(Z_p^+(w),Z_q^+(w))\le C\frac{q}{p}.
\end{equation}

\begin{lemma}
\label{lem:Zp}
Let $r$, $m$ and $p$ be such that $m$ is a positive integer, $r \ge m+1$ and $-\frac{m}{2}  \le p \le r-1$.
Let $F$ be a subspace of $\R^n$ of dimension $m$ and let $g$ be a $(-1/(r+m))$-concave density of a probability measure on $F$ such that $\int_F xg(x)dx=0$.
Then  we have
\[
d(K_{m+p}(g), Z_{\max(m,p)}^+(g)) \le c,
\]
where $c$ is a universal constant.
\end{lemma}
As in \cite{GM},  an important ingredient in the proof of the thin-shell concentration inequality is an estimate from above of the log-Lipschitz constant of the map on
$SO(n)$ :
$u \mapsto h_{k,p}(u)$.  Let ${\cal M}_n(\R)$ be the set of square $n \times n$ matrices. We equip
\[
SO(n) = \{ u \in {\cal M}_n(\R) : u^t u = Id, \mathrm{det}(u) = 1 \}
\]
with its standard invariant Riemannian metric, which we specify for concreteness on $T_{Id} SO(n)$, the tangent space at the identity element
$Id \in SO(n)$. Since $u^t u = Id$, this tangent space may be identified with the set of anti-symmetric matrices $\{ B \in {\cal M}_n(\R) : B^t + B = 0\}$.
We define the scalar product $\langle B, B \rangle =  \frac{1}{2} \mathrm{tr}(B^t B)$ on $T_{Id} SO(n)$.

\begin{prop}
\label{prop:LogLip} Let $n\ge 1$, $r> 10$ and $w$ be the $(-1/(r+n))$-concave density of a probability measure on $\R^n$ such that $\int_{\R^n}xw(x)dx=0$. Let $k$ be an integer such that $k \ge 2$, 
$2k - 1 \le n$ and $2k \le r$. Let $p$ such that  $-\frac{k}{2}\le p\le r-1$.
Denote by $L_{k,p}$ the log-Lipschitz constant of the map on $SO(n)$ : $u \mapsto h_{k,p}(u)$. Then
\[
L_{k,p} \le C \max(k,p) d(Z_{\max(k,p)}^+(w), B_2^n),
\]
where $C$ is a universal constant.
\end{prop}
\begin{proof}
For any subspace $F$ of dimension
$m$, the marginal $\pi_F(w)$ is a $(-1/(r+m))$-concave function on $F$ and from \eqref{def:K_a}, for any $a \in [0, r+m-1]$, we associate the convex body $K_a(\pi_F(w))$ in $F$.  Then the proof of Theorem 2.1 in \cite[section 2.2]{GM}
gives the upper bound:
\[
L_{k,p}\leq \max_F \{ (m+p) \, d(K_{m+p}(\pi_F(w)), B_2(F)) \}
\]
over all subspaces $F$ of dimension $m = k, k+1, 2k-1$, where $B_2(F)$ is the Euclidean unit ball in $F$.  By assumptions on $k$, we get that for these values of $m$,  $m \le 2k - 1\le n$ and $r \ge 2k \ge m+1$ and $p \ge -k /2 \ge -m/2$. Hence from Lemma \ref{lem:Zp}, we have
\[
d(K_{m+p}(\pi_F(w)), B_2(F))  \le c \, d(Z^+_{\max(m,p)}(\pi_F(w)), B_2(F)).
\]
By definition, if $X$ is the random vector with density $w$ on $\R^n$, the marginal
 $\pi_F(w)$ is the density of the projection of $X$
onto $F$, namely $P_F X$. By identification of the support functions, we have that, for any $\theta \in F$,
\[
h_{Z_p^+(\pi_F(w))}^p (\theta)= \E \langle P_FX, \theta \rangle_+^p = \E \langle X, \theta \rangle_+^p.
\]
This means that
$Z_p^+(\pi_F(w)) = P_F(Z^+_p(w))$. Since the distance to the
Euclidean ball cannot increase after projections, we conclude that
 \[
d(K_{m+p}(\pi_F(w)), B_2(F))  \le c d(Z^+_{\max(m,p)}(w), B_2^n).
\]
By equation (\ref{eq:Zp-Zq}), for $m=k,k+1, 2k-1$, one has 
\[
d(Z^+_{\max(m,p)}(w), Z^+_{\max(k,p)}(w))\le c. 
\]
This finishes the proof.
\end{proof}
We define the $q$-condition number of a random vector $X$ to be
\[
\rho_q(X) = \frac{\sup_{|\theta|_2 = 1} \left( \E \sk{X, \theta}_+^q \right)^{1/q}}{\inf_{|\theta|_2 = 1} \left( \E \sk{X, \theta}_+^q \right)^{1/q}}.
\]
Obviously, if $w$ is the density of a full-dimensional random vector $X$ in $\R^n$ then $\rho_q(X) = d(Z^+_q(w), B_2^n)$.
\begin{prop}
\label{prop:rhop}
With the same assumptions as in Proposition~\ref{prop:LogLip},
if a random vector $X$ with density $w$ is isotropic then
\[
L_{k,p} \le C \max(k,p)^2.
\]
More generally if $A$ is such that $AX$ is isotropic then
\begin{equation}
\label{eq:nonisotrope}
L_{k,p} \le C \, \max(k,p)^2 \, \|A\|\|A^{-1}\|.
\end{equation}
\end{prop}
\begin{proof}
Let $q=\max(k,p)$, then one has $1\le q\le r-1$. Using the triangular inequality we get
\[
\rho_q(X) = d(Z^+_q(w), B_2^n)\le  d(Z^+_q(w), Z^+_2(w))\ d( Z^+_2(w),B_2^n).
\]
From equation (\ref{eq:Zp-Zq}) we deduce that  $d(Z^+_q(w), Z^+_2(w))\le cq$. 
 For any $\theta \in S^{n-1}$, $\E \langle X, \theta \rangle = 0$, hence $\E\langle X, \theta \rangle_+=\E\langle -X, \theta \rangle_+$. Using this equality and equation (\ref{eq:Zp-Zq}) we deduce that 
\[
\left(\E\langle -X, \theta \rangle_+^2\right)^\frac{1}{2}\le c \E\langle -X, \theta \rangle_+=c\E\langle X, \theta \rangle_+\le c\left(\E\langle X, \theta \rangle_+^2\right)^\frac{1}{2}.
\]
Thus 
\[
\E\langle X, \theta \rangle_+^2\le \E\langle X, \theta \rangle^2=\E\langle X, \theta \rangle_+^2+\E\langle -X, \theta \rangle_+^2\le C\ \E\langle X, \theta \rangle_+^2.
\]
Hence if $X$ is isotropic we deduce that $d( Z^+_2(w),B_2^n)\le c'$. We conclude that 
\[
\rho_q(X) = d(Z^+_q(w), B_2^n)\le C'q.
\]
The conclusion follows from Proposition~\ref{prop:LogLip}.
In the general case, notice that $Z_q^+(AX)=AZ_q^+(X)$ and $d(AB_2^n,B_2^n)=\|A\|\|A^{-1}\|$, thus
\[
\rho_q(X)  \le  \rho_q(AX) \|A\|\|A^{-1}\|.
\]
\end{proof}
\begin{proof}[Proof of Theorem \ref{th:moments}]
Without loss of generality, we can assume that  $r > 32$. Indeed, if
$r \le 32$ then the statement in Theorem 1 is valid for $|p| \le c
r$ and it gives only a comparison of $(\E |X|_2^p)^{1/p}$ with $(\E
|X|_2^2)^{1/2}$ up to a constant factor. The result
is a consequence of Theorem 5.2 in \cite{AGLLOPT}.

From now on, we assume that  $r > 32$ and that $|p|\le \frac{r}{8}$.
We start by presenting a complete argument following \cite{Fleury}. This will give a complete proof of Theorem~\ref{th:moments} with a slightly weaker result. In the second part, we just indicate the needed modifications of the argument of \cite{GM} to get the complete conclusion.

In this first part, we will prove that for any $p \in
[\frac{1}{\sqrt n}, \min(cn^{1/8}, \frac{r}{8})]$
\begin{equation}
\label{eq:premierepreuve}
 (\E |X|_2^p \E |X|_2^{-p} )^{1/p} \le  1 +   \frac{C \, p}{r}   + \left(\frac{C p}{n^{1/3}}\right)^{3/5}.
\end{equation}
Assuming \eqref{eq:premierepreuve}, few elementary steps are needed to prove that for any $p$ such that $|p| \le \min(cn^{1/8}, \frac{r}{8})$,
\begin{equation}
\label{eq:approche1}
\left| \frac{(\E |X|_2^p)^{1/p}}{(\E |X|_2^2)^{1/2}} -1 \right| \le   \frac{C (1+ |p|)}{r}   + \left(\frac{C (1+|p|)}{n^{1/3}}\right)^{3/5} ,
\end{equation}
which is already enough to get a thin-shell concentration.
Indeed, for $p \ge 2$,  by H\"older inequality, we have
\[
0 \le \frac{(\E |X|_2^p)^{1/p}}{(\E |X|_2^2)^{1/2}} -1 \le \frac{(\E |X|_2^p)^{1/p}}{(\E |X|_2^{-p})^{-1/p}} -1
\]
and we conclude by \eqref{eq:premierepreuve}. For $p \le -2$, we have $|p| = -p \ge 2$ and from H\"older inequality and \eqref{eq:premierepreuve},
\[
0 \le \frac{(\E |X|_2^2)^{1/2}}{(\E |X|_2^p)^{1/p}} -1 \le \frac{(\E |X|_2^{|p|})^{1/|p|}}{(\E |X|_2^{-|p|})^{-1/|p|}} -1
\le \frac{C \, |p|}{r}   + \left(\frac{C |p|}{n^{1/3}}\right)^{3/5}.
\]
An elementary computation shows that
\[
\left| \frac{(\E |X|_2^p)^{1/p}}{(\E |X|_2^2)^{1/2}} -1 \right| \le   \frac{C|p|}{r}   + \left(\frac{C|p| }{n^{1/3}}\right)^{3/5} .
\]
For $p \in [-2, 2]$, by H\"older inequality, we have 
\[
0 \le 1 -\frac{(\E |X|_2^p)^{1/p}}{(\E |X|_2^2)^{1/2}} \le 1 - \frac{(\E |X|_2^{-2})^{-1/2}}{(\E |X|_2^2)^{1/2}}
\]
and we conclude by the previous estimate for $p = -2$. This concludes the proof of \eqref{eq:approche1}.

Let us start the proof of \eqref{eq:premierepreuve}. Let $p \in
[\frac{1}{\sqrt n}, \min(cn^{1/8}, \frac{r}{8})]$ and 
$k$ be an integer greater or equal than 2 such that $p < k \le n$. We will optimize the choice of $k$ at the end of the proof.
Recall that by $(\ref{eq:so1})$,
\[
\E |X|_2^p = \frac{\Gamma((p+n)/2)\Gamma(k/2)}{\Gamma(n/2)\Gamma((p+k)/2)} \E h_{k,p}(U),
\]
where $U$ is uniformly distributed on $SO(n)$. Using that the
function $\frac{d}{dp} \log \Gamma(p)$ is concave (see for
example the proof of Lemma \ref{lem:fonctionBeta} in the appendix),
we deduce that
\begin{equation}
\label{eq:logG}
\frac{d}{dp} \left( \frac{1}{p} \log \frac{\Gamma((p+n)/2) \Gamma(k/2)}{\Gamma((p+k)/2) \Gamma(n/2)} \right) \le 0.
\end{equation}
It follows that for any $0<p<k$,
\[
\frac{\Gamma((p+n)/2)\Gamma(k/2)}{\Gamma(n/2)\Gamma((p+k)/2)} \ \frac{\Gamma((-p+n)/2)\Gamma(k/2)}{\Gamma(n/2)\Gamma((-p+k)/2)}
\le 1.
\]
Then for all $0<p<r$ and $n \ge k > p$ we have
\begin{equation}
 \label{eq:start1}
\E |X|_2^p \E |X|_2^{-p}
 \le
\E h_{k,p}(U) \E h_{k,-p}(U).
\end{equation}
Applying log-Sobolev inequality $(\ref{eq:reverseLogSob})$ to $h_{k,p}$ and $h_{k,-p}$ we get
\begin{equation}
\label{eq:reverseHolder0}
\E h_{k,p}(U)^2  \le  e^{\frac{c \, L_{k,p}^2}{n}} \  \left(\E h_{k,p}(U) \right)^2,
\
\E h_{k,-p}(U)^2  \le  e^{\frac{c \, L_{k,-p}^2}{n}} \  \left(\E h_{k,-p}(U) \right)^2.
\end{equation}
Since $\Var f = \E f^2- (\E f)^2$ we deduce that
\begin{equation}
\label{eq:reverseHolder1}
\left\{
\begin{array}{ll}
\Var \, h_{k,p} (U) \le  \left( e^{\frac{c \, L_{k,p}^2}{n}} -1 \right) \,  \left(\E h_{k,p}(U) \right)^2 , \,
\\
\Var \, h_{k,-p} (U) \le  \left( e^{\frac{c \, L_{k,-p}^2}{n}}-1 \right) \,  \left(\E h_{k,-p}(U) \right)^2.
\end{array}
\right.
\end{equation}
By Corollary \ref{cor:logconcavity}, we know that $p \mapsto h_{k,p}(u) / B(k+p, r-p)$ is log-concave on $[-k+1, r)$ hence
\[
h_{k,p}(u) \, h_{k,-p}(u) \le \left( \frac{B(k+p,r-p)}{B(k,r)} \frac{B(k-p, r+p)}{B(k,r)} \right) \ h_{k,0}^2(u).
\]
Taking the expectation with respect to $SO(n)$, we get that
\[
\E h_{k,p}(U) h_{k,-p}(U) \le \left( \frac{B(k+p,r-p)}{B(k,r)} \frac{B(k-p, r+p)}{B(k,r)} \right) \E h_{k,0}^2(U).
\]
Since  $\E h_{k,0}(U) = 1$ we deduce from $(\ref{eq:reverseHolder0})$ that
\[
\E h_{k,0}^2(U) \le e^{\frac{c \, L_{k,0}^2}{n}}.
\]
Assume that $k$ is such that  $k \le r$ then
by $(\ref{eq:Beta2})$, we know that for $p \le  (k-1)/2$,
\[
\left( \frac{B(k+p,r-p)}{B(k,r)} \frac{B(k-p, r+p)}{B(k,r)} \right)^{1/p} \le e^{2 p \left( \frac{1}{k-1} + \frac{1}{r-1} \right)}
\le
e^{4 p \left( \frac{1}{k} + \frac{1}{r} \right)}
\]
since $k, r \ge 2$. 
Hence
\begin{equation}
\label{eq:1}
\E h_{k,p}(U) h_{k,-p}(U) \le e^{\frac{c \, L_{k,0}^2}{n} + 4 p^2 \left(\frac{1}{k} + \frac{1}{r}\right) }.
\end{equation}

Moreover
\begin{align}
\nonumber
\E h_{k,p}(U) \,  h_{k,-p}(U)   = &  \,
\E h_{k,p}(U) \  \E h_{k,-p}(U)  + \Cov(h_{k,p}(U), h_{k,-p}(U))
\\
\nonumber
\ge  \E h_{k,p}(U) \, \E h_{k,-p}(U) & - \sqrt{\Var \, h_{k,p}(U) \ \Var \, h_{k,-p}(U)}
\\
\label{eq:ma1}
\ge  \E h_{k,p}(U) \, \E h_{k,-p}(U) &
\left( 1 -  \sqrt{  \left( e^{\frac{c \, L_{k,p}^2}{n}} -1 \right) \left( e^{\frac{c \, L_{k,-p}^2}{n}} -1 \right)} \right)
\end{align}
where the last inequality follows from $(\ref{eq:reverseHolder1})$.
Assume moreover that $k$ is such that  
$2k - 1 \le n$ and $2k \le r$ then for $p \le (k-1)/2$, we can evaluate $L_{k,p}$, $L_{k,-p}$ and $L_{k,0}$ from Proposition \ref{prop:rhop}  since the assumptions are fulfilled. We get  that if $X$ is isotropic then $\max(L_{k,p}, L_{k,-p}, L_{k,0}) \le C k^2$. If $k \le c_0 n^{1/4}$ for a small enough numerical constant $c_0$,  we have
\[
 \sqrt{  \left( e^{\frac{c \, L_{k,p}^2}{n}} -1 \right) \left( e^{\frac{c \, L_{k,-p}^2}{n}} -1 \right)}
\le  c' \ \frac{k^4}{n} \le \frac{1}{10}.
\]
Combining this estimate with \eqref{eq:ma1} and
$(\ref{eq:1})$, we have proved that if $k$ is an integer such that $k \ge 2$, 
$2k - 1 \le n$, $2k \le r$, $k \le c_0 n^{1/4}$ and $2p +1\le k$  (this set of integers is not empty since $r > 32$ and $p \le r/8$) then
\[
\E h_{k,p}(U) \, \E h_{k,-p}(U)  \le \frac{e^{4 p^2 \left( \frac{1}{k} + \frac{1}{r}\right) + c \frac{k^4}{n}}}{1 - c' \ \frac{k^4}{n}} \le e^{4 p^2 \left( \frac{1}{k} + \frac{1}{r}\right) + C \frac{k^4}{n}}.
\]
For $p \le 1$, we also force $k$ to satisfy $k \le C_0 p^{1/4} n^{1/4}$. Hence taking the power $1/p$ in the last expression,
we conclude from $(\ref{eq:start1})$ that
\[
(\E |X|_2^p \E |X|_2^{-p} )^{1/p}  \le  e^{4 p \left(\frac{1}{k} + \frac{1}{r}\right) + C \frac{k^4}{pn} } \le
1 + c p \left(\frac{1}{k} + \frac{1}{r}\right) + c \frac{k^4}{pn}, 
\]
since $p/k, p/r$ and $k^4/pn$ are bounded by universal constants.
It remains to optimize the choice of $k$.
Let $p_0 = n^{-1/2}$. In this case  we choose $k=2$ and get
\begin{equation}
\label{eq:smallp}
(\E |X|_2^{p_0} \E |X|_2^{-p_0} )^{1/p_0} \le 1   +\frac{C}{\sqrt n} .
\end{equation}
If $p \ge n^{-1/2}$ we  choose $k$ to be an integer such that
$\min(r/4,(p^2n)^{1/5}) \le k \le 2 \min(r/4,(p^2n)^{1/5})$ with the restriction
$2p +1 \le k \le c n^{1/4}$ and that $k \le c p^{1/4} n^{1/4}$. 
For any $p$ such that $p_0 \le p \le \min(c \, n^{1/8},
r/8)$,  the integer $k$ satisfies $k \ge 2$, 
$2k - 1 \le n$, $2k \le r$, $k \le c_0 n^{1/4}$ and $2p +1\le k$ and we get that 
\[
 (\E |X|_2^p \E |X|_2^{-p} )^{1/p} \le  1 +   \frac{C \, p}{r}   + \left(\frac{C p}{n^{1/3}}\right)^{3/5} .
\]
This ends the proof of \eqref{eq:premierepreuve}.

In the second part, we follow the argument developed in \cite{GM} to
get a better estimate. We deal now with the case of $p$ being
positive or negative and, as already said, we can assume without
loss of generality that $r > 34$ and $|p|\le r/8$. As in
\cite{GM}, our goal is to estimate
\[
\frac{d}{dp} \log( (\E |X|_2^p)^{\frac{1}{p}} ) = \frac{d}{dp} \log((\E h_{k,p}(U))^{\frac{1}{p}}) + \frac{d}{dp} \left(\frac{1}{p}  \log \frac{\Gamma((p+n)/2)\Gamma(k/2)}{\Gamma(n/2)\Gamma((p+k)/2)}\right) ~.
\]
Most of the computation of section 3.2 in \cite{GM} is identical.
All the ingredients needed for the proof have been established and,
adapting the argument done in section 3.2 in \cite{GM}, we get
\begin{equation}
\label{eq:aprouver2}
 \frac{d}{dp} \log( (\E |X|_2^p)^{\frac{1}{p}} )
\le \frac{c}{p^2 n} (2 L_{k,p}^2 + 3 L_{k,0}^2) + \frac{C}{k-1} +
\frac{C}{r-1}.
\end{equation}
For convenience of the reader, we will shortly reproduce the proof of \eqref{eq:aprouver2} 
in the appendix.

Assume that $X$ is isotropic. For any $2 |p|\le k \le r/2$ (this set of integers is not empty since $r > 32$ and $|p| \le r/8$),  we
know by Proposition \ref{prop:rhop}, that $L_{k,p}$ and $L_{k,0}$ are smaller than $C k^2$. We get that 
\[
\frac{d}{dp} \log( (\E |X|_2^p)^{\frac{1}{p}} )  \le C \, \left(\frac{k^4}{p^2 n} + \frac{1}{k} + \frac{1}{r} \right) .
\]
We have to minimize this expression for $k$ being an integer greater
or equal than 2 and $k \in [2|p|, r/2]$. For $|p| \in [n^{-1/2},
c n^{1/3}]$, we set $k$ being an integer such that $\min(r/4, 2
(p^2 n)^{1/5}) \le k \le 2\min(r/4, 2 (p^2 n)^{1/5})$. Therefore
$k$ satisfies the restrictions and we get for any $p$ such that
$|p| \in [n^{-1/2}, c n^{1/3}]$,
\begin{equation}
\label{eq:derivee}
\frac{d}{dp} \log( (\E |X|_2^p)^{\frac{1}{p}} )  \le C \, \left(\frac{1}{(p^{2} n)^{1/5}} + \frac{1}{r} \right).
\end{equation}
After integration over $p$,  we get that  for all $p \in [n^{-1/2}, c \min(r,n^{1/3})]$
\[
\left| \log \frac{(\E |X|_2^p)^{1/p}}{(\E |X|_2^2)^{1/2}}  \right|
\le
 \frac{C\, |p-2|}{r} + \frac{C \, |p^{3/5}-2^{3/5}|}{n^{1/5}} .
\]
Since $|p^{3/5} - 2^{3/5} |\le |p-2|^{3/5}$ and all terms in the right hand side of the inequality are bounded by a universal constant,
we conclude by adjusting 

 that 
\[
\left| \frac{(\E |X|_2^p)^{1/p}}{(\E |X|_2^2)^{1/2}} -1 \right|
\le   \frac{C\, |p-2|}{r} + \left(\frac{C \, |p-2|}{n^{1/3}}\right)^{3/5},
\quad
\forall p \in [n^{-1/2},  c \min(r,n^{1/3})].
\]
Since \eqref{eq:derivee} holds only for $|p|\geq n^{-1/2}$, we use $(\ref{eq:smallp})$  to bridge the gap between $-n^{-1/2}$ and $n^{-1/2}$. Indeed, from
$(\ref{eq:smallp})$, the previous inequality for $p_0 = n^{-1/2}$ and using that $|p_0-2|=2-p_0\le 2$,   we get that for $p \in [-p_0, p_0]$,
\begin{align*}
(\E |X|_2^p)^{1/p} \ge (\E |X|_2^{-p_0})^{-1/p_0}
& \ge \frac{1}{1 + \frac{C}{\sqrt{n}}} (\E |X|_2^{p_0})^{1/p_0}
\\
& \ge \frac{1 - \frac{2C}{r} - (\frac{2C}{n^{1/3}})^{3/5}}{1 + \frac{C}{n^{1/5}} }
(\E |X|_2^2)^{1/2}.
\end{align*}
An easy adaptation of the constants leads to the conclusion of Theorem \ref{th:moments} for all $p \in [-n^{-1/2}, n^{-1/2}]$.
\\
Integrating again \eqref{eq:derivee}, we get, for $p \in [-c \min(r,n^{1/3}) , -n^{-1/2}],$
\[
\frac{(\E |X|_2^p)^{1/p}}{(\E |X|_2^{-p_0})^{-1/p_0}}\ge 
1- \frac{C\, |p+p_0|}{r} - \left(\frac{C \, |p+p_0|}{n^{1/3}}\right)^{3/5}.
\]
Using that $|p+p_0|\le|p-2|$ and the previous comparison of the moment of order $-p_0$ with the moment of order $2$ and adjusting the constants, this proves that for all $p \in [-c \min(r,n^{1/3}) , -n^{-1/2}],$
\[
\left| \frac{(\E |X|_2^p)^{1/p}}{(\E |X|_2^2)^{1/2}} -1 \right|
\le   \frac{C\, |p-2|}{r} + \left(\frac{C \, |p-2|}{n^{1/3}}\right)^{3/5}.
\]
This concludes the proof of the first part of Theorem \ref{th:moments}.

If $X$ is such that $AX$ is isotropic, we know from Proposition \ref{prop:rhop}  that for any integer $k$ such that $2 |p| \le k \le r/2$,
\[
\max(L_{k,p}, L_{k,0})  \le C k^2 \|A\| \|A^{-1}\|.
\]
The  proof is identical to the previous one replacing $n$ by $\frac{n}{\|A\|^2 \|A^{-1}\|^2}$.
\end{proof}

\begin{remark} \label{final remark}
In \cite{GM}, a preprocessing step consisted  in adding a Gaussian isotropic vector to the random vector $X$  in order to start at the very beginning with a better information on the $Z_p^+$ bodies associated to the measure. In \cite{Klartag2, Fleury}, this convolution argument played a role of regularization. It is natural to ask if such a process could be done in the situation of $s$-concave measure. Nothing is doable by adding a Gaussian vector because for $s < 0$, the new vector does not belong to any class of $s$-concave vectors.  However, for $r > n$, we can build a similar argument,  adding to $X$ a random vector $Z$ uniformly distributed on the  Euclidean ball, see also \cite{BobMad}. Since $Z$ is $(1/n)$-concave and $X$ is $(-1/r)$-concave, the new vector $Y = \frac{X+Z}{\sqrt 2}$ will be
$(- 1/(r-n))$-concave. For any $p \ge 1$, we have (see inequality $(4.7)$ in \cite{GM})
\[
\alpha_p(X) \le \alpha_{2p}(Y) \left( 2 + \alpha_{2p}(Y) \right)
\]
so that it remains to bound $\alpha_{2p}(Y)$. It is easy to see that $Y$ is such that for every $q \ge 2$ and every $\theta \in S^{n-1}$,
$\left(\E \sk{Y, \theta}_+^q \right)^{1/q} \ge c \sqrt q$.  Adapting the proof of Proposition \ref{prop:rhop}, we get  $L_{k,p} \le C \max(k,p)^{3/2}$. As in \cite{GM}, this improvement leads to the following estimate: if $r-n > 2$, then for any $p$ such that $1 \le p \le c \min(r-n, \sqrt n)$
\[
\alpha_{2p} (Y) \le  \frac{C(2p-2)}{r-n} +  \left( \frac{C(2p-2)}{\sqrt n} \right)^{1/2}.
\]
For $r > n + \sqrt n$, we recover the same thin-shell concentration as in the log-concave case.  It would be interesting to understand
in which precise sense the $s$-concave measures are close to the log-concave measures for
$s \in (-1/n, 1/n)$.
Another question is to know
what kind of preprocessing argument like in \cite{KlartagMilman} would enable to recover the small ball estimates from \cite{AGLLOPT}.
\end{remark}

\section{Appendix}
\label{sec:appendix}

\begin{proof}[Proof of Theorem \ref{th:logconcave}]
(1) This result is classical. In the symmetric case, it follows from Lemma 2.1 in \cite{MP}. The general case is similar. We provide its proof for completeness. We may assume, without loss of generality, that $\|f\|_\infty=1$. Denote  $I_p(f)=\int_0^{+\infty} t^{p-1} f(t) dt$. From H\"older inequality, the function $p\mapsto\log (I_p(f))$ is convex on its convex support, thus the domain of definition of $I_p(f)$ is an interval. Let $0<p<q$ be fixed such that $I_p(f)<+\infty$ and $I_q(f)<+\infty$. Let $a=(pI_p(f))^{1/p}$ and $\varphi(t)=t^{p-1}(f(t)-1_{[0,a]}(t))$. Notice that $\varphi\le 0$ on $[0,a]$, $\varphi\ge 0$ on $[a,+\infty)$ and $\int_0^{+\infty}\varphi(t)dt=0$. Thus
\[
I_q(f)-I_q(1_{[0,a]})=\int_0^{+\infty}t^{q-p}\varphi(t)dt=\int_0^{+\infty}(t^{q-p}-a^{q-p})\varphi(t)dt\ge 0,
\]
since the integrand is non negative on $\R_+$. We conclude that
\[
I_q(f)\ge I_q(1_{[0,a]})=\frac{a^q}{q}=\frac{1}{q}\left(pI_p(f)\right)^\frac{q}{p}.
\]

(2) Since $f$ is $(-1/\al)$-concave, there exists a convex function $\varphi: [0,\infty) \to (0,\infty)$ such that $f=\varphi^{-\alpha}$. Since  $f$ is integrable  it follows that $\varphi$ tends to $+\infty$ at $+\infty$. From the convexity of $\varphi$, one deduces that for some constant $c>0$, $\varphi(t)\ge c(1+t)$. Thus $f(t)\le (c+ct)^{-\alpha}$, for every $t\ge 0$.
Therefore, $t^{p-1}f$ is integrable for every $p<\alpha$, which means that $H_f(p)<+\infty$ for every $0< p<\alpha$.
Let $p \in (0, \al)$ and $m, M >0$. Define $g : \R_+ \to \R_+$  by $g(t) =  m  \left( 1 + \frac{t}{M} \right)^{-\al}$.
Then
\[
\int_0^{+\infty} t^{p-1} g(t) dt = m M^p \int_0^{+\infty} v^{p-1} (1+v)^{-\al}dv =   m M^p  B(p, \al-p).
\]
Thus $H_g(p)=mM^p$, which implies that $\log(H_g)$ is affine on $(0, \al)$.
Take $0<a<b<c<\al$. Let $\lambda\in[0,1]$ be such that $b=(1-\lambda)a+\lambda c$. Choose $m$ and $M$ such that
$m M^a = H_f(a)$  and $m M^b  = H_f(b)$ so that $H_g(a)=H_f(a)$ and $H_g(b)=H_f(b)$.
If we prove that
\begin{equation}
\label{eq:technical}
    \int_{0}^{+\infty} t^{c-1} (g-f)(t)  dt \geq 0,
\end{equation}
that is $H_g(c)\ge H_f(c)$, then using that $\log(H_g)$ is affine, we will  deduce that
\[
    H_f(b) = H_g(b) = H_g(a)^{1-\lambda} H_g(c)^{\lambda} \geq H_f(a)^{1-\lambda} H_f(c)^{\lambda}
\]
and this will prove the log-concavity of $H$ on $(0,\alpha)$. If $f=g$ then (\ref{eq:technical}) is satisfied so that in the following we assume that the function  $h:=g-f\not\equiv 0$. Let
\[
    H_1(t) = \int_t^{+\infty} s^{a-1} h(s) ds
\quad \hbox{ and} \quad
    H_2(t) = \int_t^{+\infty}  s^{b-a-1} H_1(s) ds.
\]
Since $h(t)=O(t^{-\alpha})$ at infinity, we deduce that $H_1(t)=O(t^{a-\alpha})$ and $H_2(t)=O(t^{b-\alpha})$.
We have $\int_0^{+\infty} {t^{a-1} h(t)} dt  = 0$
thus $H_1(\infty) = H_1(0) = 0.$
Obviously $H_2(\infty)=0$. We also observe
\begin{align*}
    0 & =  \int_{0}^{+\infty}{t^{b-1} h(t)}{dt} = \int_{0}^{+\infty}{t^{b-a} t^{a-1} h(t)}{ dt}= - \int_{0}^{+\infty}{t^{b-a} H_1'(t)}{dt}  \\
    & =[t^{b-a}H_1(t)]_0^{+\infty}+ (b-a) \int_{0}^{+\infty}{t^{b-a-1}H_1(t)}{ dt}= (b-a) H_2(0),
\end{align*}
whence $H_2(\infty)=H_2(0)=0$. Since  $\int_{0}^{+\infty}{t^{b-a-1}H_1(t)}{ dt}=0$ and $H_1\not\equiv 0$,
the function $H_1$ has at least one change of sign.
Moreover, using that $H_1(0)=H_1(\infty)=0$, we deduce that $H_1'$ and therefore $h$ has at least two sign changes. Since $h=g-f$ has the same sign as  $f^{-\al}-g^{-\al}$ which is convex, it cannot have more than two sign changes. Thus it has exactly two sign changes at some $0<t_1<t_2$. Moreover, from the convexity of $f^{-\al}-g^{-\al}$, the sign of $h$ has to be negative on $(t_1,t_2)$ and positive on $(0,t_1)$ and $(t_2,+\infty)$.
From an easy study of the function $H_2$, we deduce that
%
%\bigskip
%\noindent
%\begin{tikzpicture}
%\tkzTabInit{$t$/1,$h(t)$/1,$H_1'(t)$/1, $H_1(t)$/2, $H_2'(t)$/1, $H_2(t)$/2}{$0$,$t_1$, ,$t_2$, $+\infty$}
%\tkzTabLine{,+, z, ,- , ,z, +}
%\tkzTabLine{,-, z, ,+ , ,z, -}
%\tkzTabVar{+/$0$,-/, R/, +/, -/$0$}
%\tkzTabVal{2}{4}{0.5}{}{$0$}
%\tkzTabLine{, ,+, ,z , , -}
%\tkzTabVar{-/$0$,R/, +/, R/, -/$0$}
%\end{tikzpicture}
%\\
$H_2 \geq 0$. Therefore, using that $H_1(0)=H_1(\infty)=H_2(0)=H_2(\infty)=0$, we get
\begin{align*}
    \int_{0}^{+\infty}{t^{c-1} h(t)}{dt} & = \int_{0}^{+\infty}{t^{c-a}t^{a-1} h(t)}{dt} = -\int_{0}^{+\infty}{t^{c-a}H_1'(t)}{dt} \\
    &= [-t^{c-a}H_1(t)]_0^{+\infty}+ (c-a) \int_{0}^{+\infty}  {t^{c-a-1}H_1(t)}{dt}  \\
    &= (c-a) \int_{0}^{+\infty}{t^{c-b}t^{b-a-1}H_1(t)}{dt} \\
    &= (c-a)[-t^{c-b}H_2(t)]_0^{+\infty}+ (c-a)(c-b) \int_{0}^{+\infty}{t^{c-b-1}H_2(t)}{dt}\\
    &= (c-a)(c-b) \int_{0}^{+\infty}{t^{c-b-1}H_2(t)}{dt} \geq 0.
\end{align*}
This proves $(\ref{eq:technical})$ and establish the log-concavity of $H_f$ on $(0,\alpha)$. To get it on $[0,\alpha)$, it is enough to prove that $H_f$ is continuous at $0$. This follows from the observation that
\[
B(p,\al-p)   \mathrel{\mathop{\sim}\limits_{p\to0}} \Gamma(p)
\mathrel{\mathop{\sim}\limits_{p\to0}}\frac{1}{p}
\quad \hbox{ thus} \quad
 H_f(p)\mathrel{\mathop{\sim}\limits_{p\to0}}p\int_0^{+\infty}t^{p-1}f(t)dt.
\]
And it is classical that, for a continuous function $f$, the right hand side term tends to $f(0)$ when $p\to 0$.
\end{proof}

\begin{proof}[Proof of Lemma \ref{lem:fonctionBeta}]
Equation $(\ref{eq:Beta1})$ follows easily from the classical bounds for the Gamma function (see \cite{Artin}), valid for $x\ge 1$:
\[
\sqrt{2\pi}x^{x-\frac{1}{2}}e^{-x} \le \Gamma(x) \le \sqrt{2\pi}x^{x-\frac{1}{2}}e^{-x+\frac{1}{12}}.
\]
For equation $(\ref{eq:Beta2})$, we write that
\[
\frac{B(k+p, r-p)}{B(k,r)} = \frac{\Gamma(k+p) \Gamma(r-p)}{\Gamma(k) \Gamma(r)}.
\]
Denote $G(p)=\log \Gamma(p)$, for $p> 0$. We know that $G''(p)=\sum_{i\ge 0}1/(p+i)^2$ hence $G''$ is non-increasing and  $0 \le G''(p)\le 1/(p-1)$, for $p>1$.
Denote $F_k(p)=\frac{G(k+p)-G(k)}{p}$, for $k>0$ and $p>-k$. We have $F_k(p)=\int_0^1G'(k+up)du$. Using that  $G''$ is non-increasing,
we get that for $k>1$ and $p\ge -(k-1)/2$,
\[
F_k'(p)=\int_0^1G''(k+up)udu\le G''\left(\frac{k+1}{2}\right) \int_0^1 udu=\frac{1}{2}G''\left(\frac{k+1}{2}\right)\le \frac{1}{k-1}
\]
and $F_k'(p)\ge 0$. Therefore, if $k>1$, $r>1$ and $-\frac{k-1}{2}\le p\le \frac{r-1}{2}$ then
\begin{eqnarray*}
0\le\frac{d}{dp} \left( \frac{1}{p} \log \frac{B(k+p, r-p)}{B(k,r)} \right)
& = &\frac{d}{dp} (F_k(p)-F_r(-p))
\\
& = & F_k'(p)+F_r'(-p)\le \frac{1}{k-1}+\frac{1}{r-1}.
\end{eqnarray*}
\end{proof}

\begin{proof}[Proof of Lemma \ref{lem:Zp}]
We present here a similar proof than in the appendix of \cite{GM}.
Applying Corollary \ref{cor:geoKr} to $w=g$, $n=m$, $\alpha=r+m$,  we deduce that, for $\frac{m}{2}\le a\le b\le r+m-1$, one has
\[
\left(\frac{g(0)}{\|g\|_{\infty}}\right)^{\frac{1}{a} - \frac{1}{b}} K_a(g) \subset K_b(g) \subset \frac{(bB(b, r+m -b))^{1/b}}{(aB(a, r+m - a))^{1/a}} \ K_a(g).
\]
From Lemma \ref{lem:fonctionBeta}, we have
\[
\frac{(bB(b, r+m -b))^{1/b}}{(aB(a, r+m - a))^{1/a}}
\le c \frac{b}{a} 
\]
Moreover since $\int xg(x)dx=0$, from Lemma~7.2 of \cite{AGLLOPT}, one has
\[
\frac{g(0)}{\|g\|_\infty}\ge\left(\frac{r-1}{r+m-1}\right)^{r+m} \ge e^{-2m}.
\]
Using that $\frac{1}{a}-\frac{1}{b}\le \frac{1}{a}\le \frac{2}{m}$, we deduce that 
$\left(\frac{g(0)}{\|g\|_\infty}\right)^{\frac{1}{a}-\frac{1}{b}} \ge e^{ - 4}$. We conclude that for $\frac{m}{2}\le a\le b\le r+m-1$, one has
\begin{equation}
\label{eq:Ka-Kb}
e^{-4} K_a(g) \subset K_b(g) \subset c\frac{b}{a} \ K_a(g).
\end{equation}
By integration in polar coordinates, it is well known \cite{Paouris1} (see also \cite{GuedonPologne}) that we have the following relation between the $Z_q^+$-bodies associated with $g$ and the $Z_q^+$-bodies associated with one of the convex bodies $K_a(g)$: for any $ 0< q < r$
\begin{equation}
\label{eq:egalite1}
Z_q^+(g)  =  g(0)^{1/q} Z_{q}^+(K_{m+q}(g)),
\end{equation}
where for any body $K$, $Z_q^+(K)$ denotes the convex body whose support function is defined by 
\[\forall \theta \in \R^m,  \quad
h_{Z_q^+(K)}(\theta)= \left(\int_{K}{}{\sk{x,\theta}_+^q}{dx}\right)^\frac{1}{q}.
\]
Let $\theta \in \R^m$ and $K$ be a convex body containing $0$. From Berwald's inequalities \cite{Berwald} applied to $K \cap \{\sk{x,\theta} \ge 0\}$ and the function $x\mapsto\langle x,\theta\rangle_+$ which is concave on $K \cap \{\sk{x,\theta} \ge 0\}$, the function
\[
p\mapsto
\left(\frac{\int_{K}{}{\sk{x,\theta}_+^p}{dx}}{mB(p+1,m)\vol(K \cap \{\sk{x,\theta} \ge 0\})}\right)^\frac{1}{p}
\]
is decreasing.
Observe that for every $\theta\in\R^m$,
$\lim_{p\to\infty}\left(\int_{K}{}{\sk{x,\theta}_+^p}{dx}\right)^\frac{1}{p}=h_K(\theta)$ and that
\[
(mB(p+1,m))^\frac{1}{p}= \left(m\int_0^1u^p(1-u)^{m-1}du\right)^\frac{1}{p}\mathop\to_{p\to+\infty}1.
\]
We deduce that
\[
\left(\frac{\int_{K}{}{\sk{x,\theta}_+^q}{dx}}{mB(q+1,m)\vol(K \cap \{\sk{x,\theta} \ge 0\})}\right)^\frac{1}{q}\ge h_K(\theta).
\]
Note also that $\int_{K}{}{\sk{x,\theta}_+^q}{dx}\le h_K(\theta)^q\vol(K \cap \{\sk{x,\theta} \ge 0\})$ and that
$mB(q+1,m)=qB(q,m+1)$. Therefore 
\begin{equation}
\label{eq:mainZp}
h_K(\theta)
\ge
\frac{h_{Z_q^+(K)}(\theta)}{\vol(K \cap \{\sk{x,\theta} \ge 0\})^{1/q}}
\ge
\left( qB(q,m+1) \right)^{1/q}  h_K(\theta).
\end{equation}
Now we establish that for $q = \max(p,m)$
\begin{equation}
\label{eq:mainZpKp}
d(K_{m+q}(g), Z_{q}^+(g)) \le c.
\end{equation}
Take  $K=K_{m+q}(g)$. By Lemma \ref{lem:fonctionBeta},  for any $q \ge m \ge 1$, $\left( q B(q, m+1) \right)^{1/q} \ge c q /(m+q+1) \ge c/3$ and we deduce from $(\ref{eq:mainZp})$ that for every 
$\theta \in \R^n$, 
\[
h_{K_{m+q}(g)} (\theta)
\ge
\frac{h_{Z_q^+(K_{m+q}(g))}(\theta)}{\vol(K_{m+q}(g) \cap \{\sk{x,\theta} \ge 0\})^{1/q}}
\ge
\frac{c}{3} \ h_K(\theta).
\]
where $c$ is a universal constant. Together with  $(\ref{eq:egalite1})$, we conclude that
\begin{eqnarray}
\nonumber
d(K_{m+q}(g), Z_q^+(g))=d(K_{m+q}(g), Z_q^+(K_{m+q}(g)))
\\
\label{equ: inegalite first part}
\le c 
\frac{\sup_{\theta \in \R^n} \vol(K_{m+q}(g) \cap \{\sk{x,\theta} \ge 0\})^{1/q}}{\inf_{\theta \in \R^n} \vol(K_{m+q}(g) \cap \{\sk{x,\theta} \ge 0\})^{1/q}}
\end{eqnarray}
for a universal constant $c$. Applying (\ref{eq:Ka-Kb}) for $a=m+1$ and $b=m+q$, we get 
\[
e^{-4} K_{m+1}(g) \subset K_{m+q}(g) \subset c\frac{m+q}{m+1} \ K_{m+1}(g).
\]
 Since 
$q \ge m$ and $\left(\frac{m+q}{m+1}\right)^{m/q} \le e$, we get from \eqref{equ: inegalite first part}
\[
d(K_{m+q}(g), Z_q^+(g))\le C \
\frac{\sup_{\theta \in \R^n} \vol(K_{m+1}(g) \cap \{\sk{x,\theta} \ge 0\})^{1/q}}{\inf_{\theta \in \R^n} \vol(K_{m+1}(g) \cap \{\sk{x,\theta} \ge 0\})^{1/q}}
\]
for a universal constant $C$. 
Since $g$ has its barycenter at the origin then $K_{m+1}(g)$ has also its barycenter at the origin and we deduce from a classical result of Gr\"unbaum \cite{Grunbaum} that there exists a universal constant $c$ for which
\[
\frac{\sup_{\theta \in \R^n} \vol(K_{m+1}(g) \cap \{\sk{x,\theta} \ge 0\})^{1/q}}{\inf_{\theta \in \R^n} \vol(K_{m+1}(g) \cap \{\sk{x,\theta} \ge 0\})^{1/q}} \le (e-1)^{1/q} \le e-1.
\]
And \eqref{eq:mainZpKp} is proved.

To conclude the proof of the Lemma, it is enough to establish that $d(K_{m+q},K_{m+p})\le c$, where $q=\max (m,p)$. For $q = p$, this is obvious so we may assume that $q=m\ge p$. Then $\frac{m}{2}\le m+p\le m+q=2m$ and using equation (\ref{eq:Ka-Kb}) for $a=m+p\le b=2m$, we deduce that
\[
d(K_{m+p}(g), K_{2m}(g)) \le ce^4 \frac{2m}{m+p}\le 4ce^4.
\]
\end{proof}

\begin{proof}[Proof of inequality \eqref{eq:aprouver2}]
Our goal is to estimate
\[
\frac{d}{dp} \log\left( \left(\E |X|_2^p\right)^{\frac{1}{p}} \right) = \frac{d}{dp} \log\left(\left(\E h_{k,p}(U)\right)^{\frac{1}{p}}\right) + \frac{d}{dp} \left(\frac{1}{p}  \log \frac{\Gamma((p+n)/2)\Gamma(k/2)}{\Gamma(n/2)\Gamma((p+k)/2)}\right) ~.
\]
As already mentioned in \eqref{eq:logG}, by concavity of $p \mapsto \frac{d}{dp} \log \Gamma(p)$, we have 
\[
\frac{d}{dp} \left(\frac{1}{p}  \log \frac{\Gamma((p+n)/2)\Gamma(k/2)}{\Gamma(n/2)\Gamma((p+k)/2)}\right)  \le 0.
\]
We use the following convention: let $(\Omega, \mu)$ be a measurable space, for any measurable function $f: \Omega \to \R^+$, we set
\[
\E_\mu(f) = \int f d\mu \quad \hbox{and} \quad
\Ent_\mu (f)  = \E_\mu(f\log f) - \E_\mu (f) \log (\E_\mu(f)).
\]
Let $w$ be the density of the distribution of $X$ on $\R^n$. Since $X$ is $(-1/r)$-concave, $w$ is $(-1/(r+n))$-concave on $\R^n$.
To any fixed $u \in SO(n)$, we associate the measure $\mu_u$ on $\R^+$ with density 
\[
t\mapsto|S^{k-1}| t^{k-1} \pi_{u(E_0)}w(t u(\theta_0))
\]
so that 
\[
h_{k,p}(u) = |S^{k-1}| \int_{0}^\infty t^{p+k-1} \pi_{u(E_0)} w(t u(\theta_0)) dt = \E_{\mu_u} (t^p). 
\]
Define also $\mu_{k,p}$ the measure on $\R^+$ 
with density 
\[
t\mapsto|S^{k-1}| t^{k-1} \E\pi_{U(E_0)}w(t U(\theta_0)).
\]
Then $\E h_{k,p} (U) = \E_U \E_{\mu_U} (t^p) = \E_{\mu_{k,p}} (t^p)$. Since $w$ is a density of probability, $\mu_{k,p}$ is a probability measure on $\R^+$. A classical fact, verified by direct computation, is that
\[
\frac{d}{dp} \log \left( \left(\E_\mu(f^p) \right)^{1/p} \right) = \frac{1}{p^2} \frac{\Ent_\mu(f^p)}{\E_\mu(f^p)}.
\]
Therefore 
\begin{align}
\nonumber
\frac{d}{dp} \log\left(\left(\E h_{k,p}(U)\right)^{\frac{1}{p}}\right) 
& =
\frac{d}{dp} \log\left(\left(\E_{\mu_{k,p}} (t^p)\right)^{\frac{1}{p}}\right) 
\\
\label{eq:decomposition}
& =
\frac{1}{p^2} \frac{\Ent_{\mu_{k,p}}(t^p)}{\E_{\mu_{k,p}}(t^p)}
=
\frac{1}{p^2} \frac{\Ent_{\mu_{k,p}}(t^p)}{\E h_{k,p}(U)}.
\end{align}
The numerator can be decomposed into two terms:
\[
\Ent_{\mu_{k,p}}(t^p) = \E_U \Ent_{\mu_U} (t^p) + \Ent_U \E_{\mu_U} (t^p) 
= \E_U  \Ent_{\mu_U} (t^p) + \Ent_U h_{k,p}(U).
\]
To control the second term, we use the log-Sobolev inequality \eqref{eq:LogSob-vrai}:
\begin{equation}
\label{eq:secondterm}
\frac{1}{p^2} \frac{\Ent_U h_{k,p}(U)}{\E h_{k,p}(U)}
\le \frac{c}{p^2 n} \frac{\E \left( | \nabla \log h_{k,p} |^2 (U) h_{k,p}(U) \right) }{\E h_{k,p}(U)}
\le \frac{c L_{k,p}^2}{p^2 n}.
\end{equation}
To control the first term, we start by observing that for a fixed $u \in SO(n)$, 
\begin{align*}
& \frac{1}{p^2} \frac{ \Ent_{\mu_u} (t^p)}{\E_{\mu_u}(t^p)} 
 = 
\frac{d}{dp} \log \left( \left(\E_{\mu_u}(f^p) \right)^{1/p} \right)  = \frac{d}{dp} \left(\frac{1}{p} \log  h_{k,p}(u)\right) 
\\
& =
\frac{d}{dp}  \frac{1}{p} \left(  \log 
\frac{ h_{k,p}(u)  }{B(p+k, r -p )  } -  \log \frac{ h_{k,0}(u)  }{B(k, r)}  + 
\log \frac{B(p+k, r -p )  }{B(k,r)} + \log h_{k,0}(u) \right).
\end{align*}
By Corollary \ref{cor:logconcavity}, the map
$p \mapsto \frac{ h_{k,p}(u)  }{B(p+k, r -p )}$
is log-concave on $(-k+1, r)$. This implies that 
\[
\frac{d}{dp} \frac{1}{p} \left( \log \frac{ h_{k,p}(u)  }{B(p+k, r
-p)} - \log \frac{ h_{k,0}(u)  }{B(k, r)} \right)  \le 0.
\]
We know from Lemma \ref{lem:fonctionBeta} that, for all $p \in[-\frac{k-1}{2}, \frac{r-1}{2}]$,
\[
\frac{d}{dp} \left( \frac{1}{p} \log \frac{B(k+p, r-p)}{B(k,r)}
\right) \le C \left( \frac{1}{k-1} + \frac{1}{r-1} \right).
\]
Therefore, for any fixed $u \in SO(n)$, 
\[
\frac{1}{p^2}
\Ent_{\mu_u} (t^p)  \le C  h_{k,p}(u)  \left( \frac{1}{k-1} + \frac{1}{r-1} \right) - \frac{1}{p^2} \, h_{k,p}(u) \log h_{k,0}(u).
\]
Integrating over $u \in SO(n)$, we deduce that
\begin{equation}
\label{eq:entropie2}
\frac{1}{p^2} \frac{\E \, \Ent_{\mu_U} (t^p)}{\E h_{k,p}(U)}
\le C  \left( \frac{1}{k-1} + \frac{1}{r-1} \right)  + \frac{1}{p^2} \frac{\E h_{k,p}(U) \log (h_{k,0}(U)^{-1})}{\E h_{k,p}(U)}.
\end{equation}
From Jensen and H\"older inequalities, 
\begin{align*}
\frac{\E (h_{k,p}(U) \log h_{k,0}(U)^{-1})}{\E h_{k,p}(U)}
& \le \log \left(\frac{\E (h_{k,p}(U)  h_{k,0}(U)^{-1})}{\E h_{k,p}(U)}\right)
\\
\le &
\log \left( \frac{(\E h_{k,p}(U)^2)^{1/2}}{\E h_{k,p}(U)} \right) + \log \left( (\E (h_{k,0}(U)^{-2}))^{1/2} \right).
\end{align*}
From \eqref{eq:reverseLogSob}, the first term is upper bounded by $ \frac{c}{n} L_{k,p}^2$. For the second term, we first use \eqref{eq:reverseLogSob} with $f=h_{k,0}^{-1}$, $q=2$ and $r=0$, then we use \eqref{eq:reverseLogSob} again with $f=h_{k,0}$, $q=1$ and $r=0$. Since $\E h_{k,0}(U) = \E_{\mu_{k,0}}(1) = 1$,
we deduce that this term is bounded by $ \frac{3c}{n} L_{k,0}^2$. Combining this last inequality with 
\eqref{eq:entropie2}, \eqref{eq:secondterm} and \eqref{eq:decomposition}, we conclude that 
\[
 \frac{d}{dp} \log( (\E |X|_2^p)^{\frac{1}{p}} )
\le \frac{c}{p^2 n} (2 L_{k,p}^2 + 3 L_{k,0}^2) + \frac{C}{k-1} +
\frac{C}{r-1}.
\]
\end{proof}

\subsection*{Acknowledgements}
The research of the three authors was partially supported by the ANR project GeMeCoD, ANR 2011 BS01 007 01.

\address

\end{document}